\documentclass[11pt]{amsart}
\usepackage{geometry,verbatim}                
\usepackage{graphicx}
\usepackage{amssymb}
\usepackage{epstopdf}

\usepackage{float, caption, amsmath, amsthm, empheq, subfig, gensymb, tcolorbox, stmaryrd, tikz-cd, mathrsfs, stackengine, amscd, bm, mathtools}
\newcommand{\D}{\mathcal D}
\newcommand{\Z}{\mathbb Z}
\renewcommand{\P}{\mathbb P}
\newcommand{\E}{\mathbb E}
\newcommand{\F}{\mathcal F}
\newcommand{\ex}{\textup{ex}}
\newcommand{\floor}[1]{\left\lfloor #1 \right\rfloor}

\newtheorem{thm}{Theorem}[section]

\newtheorem{cla}[thm]{Claim}

\theoremstyle{plain}

\newtheorem{definition}[thm]{Definition}

\newtheorem{lemma}[thm]{Lemma}

\newtheorem{theorem}[thm]{Theorem}
\newtheorem{conjecture}[thm]{Conjecture}

\newtheorem{rem}[thm]{Remark}

\title{Lower bounds for the Tur\'an densities of daisies}
\author[David Ellis and Dylan King]{David Ellis\** and Dylan King\textsuperscript{\textdagger}}
\date{June 2022.\\ \small \**School of Mathematics, University of Bristol, UK. \texttt{david.ellis@bristol.ac.uk}. \\ \textsuperscript{\textdagger}School of Mathematics, University of Bristol, UK. \texttt{gs21934@bristol.ac.uk}. Supported by a Marshall Scholarship.}                             

\begin{document}
\maketitle
\begin{abstract}
For integers $r \geq 2$ and $t \geq 2$, an $r$-uniform {\em $t$-daisy} $\D^t_r$ is a family of $\binom{2t}{t}$ $r$-element sets of the form
$$\{S \cup T \ :  T\subset U, \ |T|=t  \}$$
for some sets $S,U$ with $|S|=r-t$, $|U|=2t$ and $S \cap U = \emptyset$. It was conjectured by Bollob\'as, Leader and Malvenuto in \cite{BOLLOBS2011} (and independently by Bukh) that the Tur\'an densities of $t$-daisies satisfy $\lim\limits_{r \to \infty} \pi(\D_r^t) = 0$ for all $t \geq 2$ (an equivalent conjecture was made independently by Johnson and Talbot \cite{JohnsonTalbot}). This has become a well-known problem, and it is still open for all values of $t$. In this paper, we give lower bounds for the Tur\'an densities of $r$-uniform $t$-daisies. To do so, we introduce (and make some progress on) the following natural problem in additive combinatorics: for integers $m \geq 2t \geq 4$, what is the maximum cardinality $g(m,t)$ of a subset $R$ of $\mathbb{Z}/m\mathbb{Z}$ such that for any $x \in \mathbb{Z}/m\mathbb{Z}$ and any $2t$-element subset $X$ of $\mathbb{Z}/m\mathbb{Z}$, there are $t$ distinct elements of $X$ whose sum is not in the translate $x+R$? This is a slice-analogue of the extremal Hilbert cube problem considered in \cite{Cilleruelo2017} and \cite{GUNDERSON1998}.
\end{abstract}

\section{Introduction}
For integers $r \geq 3$ and $t \geq 2$, an $r$-uniform {\em $t$-daisy} $\D^t_r$ is a collection of $\binom{2t}{t}$ $r$-element sets of the form
$$\{S \cup T \ :  T\subset U, \ |T|=t  \}$$
for some sets $S,U$ with $|S|=r-t$, $|U|=2t$ and $S \cap U = \emptyset$. As usual, for an integer $n \geq 3$ we write $\ex(n,\D^t_r)$ for the $n$th {\em Tur\'an number} of $\D_r^t$, i.e.\ the maximum possible cardinality of a family of $r$-element subsets of $\{1,2,\ldots,n\}$ which is $\D^t_r$-free, and we write
\begin{equation*}
\pi(\D^t_r) = \lim_{n \to \infty}\frac{\ex(n,\D^t_r)}{\binom{n}{r}}
\end{equation*} 
for the {\em Tur\'an density} of $\D^t_r$.

In the case $r=2$, we have $\pi(\D_2^2)=2/3$ by Tur\'an's theorem for $K_4$'s. The first unknown case occurs when $r=3$; in this case, Bollob\'as, Leader and Malvenuto show in \cite{BOLLOBS2011} that $\pi(\D_3^2)\geq\frac{1}{2}$, by taking the complement of the Fano plane, blowing up and iterating, and conjecture that in fact equality holds. In Proposition 4.3 of \cite{Falgas2013}, Falgas-Ravry and Vaughan use the semidefinite programming approach developed by Razborov to show that $\pi(\D_3^2)\leq 0.504081$. For larger $t \geq 2$ and $r \geq 3$, even less is known concerning $\pi(\D_r^t)$. The following conjecture was made by Bollob\'as, Leader and Malvenuto in \cite{BOLLOBS2011} (and independently by Bukh, see \cite{BOLLOBS2011}).
\begin{conjecture}[Bollob\'as-Leader-Malvenuto / Bukh]
\label{conj:daisy}
For all $t \geq 2$, $\lim\limits_{r \to \infty} \pi(\D^t_r)=0.$
\end{conjecture}
This is still open even for $t=2$. Johnson and Talbot independently made an equivalent conjecture in \cite{JohnsonTalbot}, which we now describe. To state it, we need the (standard) definition of a subcube of the Boolean cube.
\begin{definition}
    For $n,d \in \mathbb{N}$ with $1 \leq d \leq n$, a {\em $d$-dimensional subcube} of the $n$-dimensional Boolean cube $\{0,1\}^n$ is a subset of $\{0,1\}^n$ of the form
    \begin{equation*}
        \{x \in \{0,1\}^n \ : \ x_i = a_i \ \forall \ i \in I \}
    \end{equation*}
    for some set $I \in \binom{[n]}{n-d}$ and values $a_i$ such that $a_i \in \{0,1\}$ for each $i \in I$. (The elements of the set $I$ are called the {\em fixed coordinates} of the subcube; the elements of $[n] \setminus I$ are called the {\em moving coordinates}.) 
\end{definition}
\noindent (Here, and henceforth, we write $[n]: = \{1,2,\ldots,n\}$ for the standard $n$-element set.)
\begin{conjecture}[Johnson-Talbot]\label{conj:TJ}
    Let $d \geq 2$ and $\delta \in (0,1]$. Then for $n$ sufficiently large depending on $d$ and $\delta$, and any set $A \subset \{0,1\}^n$ with $|A| \geq \delta 2^n$, there exists a $d$-dimensional subcube $\mathcal{C}$ with $|A \cap \mathcal{C} | \geq \binom{d}{\floor{d/2}}$. 
\end{conjecture}
\noindent It is easy to verify Conjecture \ref{conj:TJ} for $d=2$ and $d=3$, but it remains open for all $d \geq 4$. It is easy to see that Conjectures \ref{conj:daisy} and \ref{conj:TJ} are equivalent for $d=2t$; the reader is referred to \cite{BOLLOBS2011} for details.

The value $\pi(\D_r^t)$ is clearly nondecreasing in $t$ for fixed $r$, since an $r$-uniform family that is free of $t$-daisies is also free of $t'$-daisies for all $t'>t$. It is also easy to see (by averaging over links of vertices) that the value $\pi(\D_r^t)$ is nonincreasing in $r$ for each fixed $t$.

In \cite{BOLLOBS2011}, a lower bound of $\pi(\D_r^2)\geq r!/r^r $ is observed; this comes from considering the $r$-partite $r$-uniform hypergraph on $[n]$ with parts of sizes as equal as possible. However, this lower bound is exponentially small in $r$. In this paper, we obtain the following improved lower bound, which is polynomial in $r$, and linear when $t=2$, using an additive-combinatorial construction. We also raise a question in additive combinatorics which may be of interest in its own right.
\begin{theorem}\label{thm:main_result}
We have $ \pi(D_r^2) = \Omega(1/r)$. Furthermore, for each $t \geq 3$, we have 
\begin{equation*}
    \pi(D_r^t) \geq r^{-\frac{4t-2}{\binom{2t}{t}-1}-O(1/\sqrt{\log r})}.
\end{equation*}
\end{theorem}

Our proof of Theorem \ref{thm:main_result} relies upon the existence of a subset of $\mathbb{Z}_m:=\mathbb{Z}/m\mathbb{Z}$ that avoids a certain additive structure, which we define now.
\begin{definition}\label{def:g}
For positive integers $m,t \geq 2$ with $m \geq 2t$, let $g(m,t)$ denote the maximum possible size of a subset $R \subset \Z_m$ such that for any $x_0\in \Z_m$ and any $(2t)$-element subset $X$ of $\mathbb{Z}_m$, there are $t$ distinct elements of $X$ whose sum is not contained in $R-x_0$, i.e.
\begin{equation*}
\left\{x_0+\sum_{x \in T}x \ : \ T\subset X, \ |T|=t\right\} \not \subset R.
\end{equation*}
\end{definition}
\noindent For brevity, given a set $X \in \binom{\mathbb{Z}_m}{2t}$ we write 
\begin{equation*}
C(X):=\left\{\sum_{x \in T}x \ : \ T\subset X, \ |T|=t\right\}
\end{equation*}
for the set of sums of $t$ distinct elements of $X$; $g(m,t)$ is the maximum size of a subset of $\mathbb{Z}_m$ containing no translate of $C(X)$ for any $|X|=2t$.
\newline

The function $g(m,t)$ is related to a question raised by Gunderson and R\"{o}dl in \cite{GUNDERSON1998}, concerning Hilbert cubes.
\begin{definition}\label{def:hcube}
If $B$ is a ring, the {\em $d$-dimensional Hilbert cube} generated by $x_1,\dots,x_d \in B$ is the set
\begin{equation*}
\left\{\sum_{i \in I}x_i \ : \ I \subset  \{1,2,\dots,d\} \right\} \subset B.
\end{equation*}
\end{definition}
Gunderson and R\"{o}dl considered large sets of integers which do not contain any translate of a Hilbert cube (working over the integers, i.e., the $B=\mathbb{Z}$ case of Definition \ref{def:hcube}). In particular, they prove the following (Theorem 2.3 and Theorem 2.5 of \cite{GUNDERSON1998}).
\begin{theorem}\label{thm:gunder}
For each integer $d \geq 3$, there exists $c_d > 0$ such that any set of integers $A \subset [m]$ with $|A| \geq c_d (\sqrt{m}+1)^{2-\frac{1}{2^{d-2}}}$ contains a translate of a $d$-dimensional Hilbert cube. Furthermore, for all $m$ there exists a set of integers $A \subset [m]$ with $|A| \geq m^{1-\frac{d}{2^d-1}-O(1/\sqrt{\log m})}$ that does not contain any translate of a Hilbert cube.
\end{theorem}

We consider the case $R = \mathbb{Z}_m$, but for our purposes, the structural differences between $[m]$ and $\Z_m$ will not be particularly important. Estimating $g(m,t)$ is a natural variant of the Gunderson-R\"odl problem, where we avoid only the middle slice of a Hilbert cube of dimension $2t$. We make a small but important change in that, in Definition \ref{def:g}, we require that $X$ be composed of $2t$ distinct elements, while a Hilbert cube may even have $x_1=x_2\dots=x_d$. In this case, the Hilbert cube is a $(d+1)$-element arithmetic progression, but $C(X)$ is a singleton. 

We obtain the following bounds on $g(m,t)$.
\begin{theorem}
\label{thm:gmt_lb}
For all $t \geq 3$ and $m \geq 4$ we have
$$g(m,t) \geq m^{1-\frac{2t}{\binom{2t}{t}-1}-O(1/\sqrt{\log m})},$$ and if furthermore $m$ is prime, then
$$g(m,t) \geq m^{1-\frac{2t-1}{\binom{2t}{t}-1}-O(1/\sqrt{\log m})}.$$ For $t=2$ and $m \geq 64$ we have $g(m,2) \geq \sqrt{m}/8$.
\end{theorem}
\begin{theorem}
\label{thm:gmt_ub}
For each $t \geq 2$ and all $m$ sufficiently large depending on $t$, we have 
$$g(m,t) \leq 4^{1-1/2^{2t}}(\sqrt{m}+\sqrt{2t})^{2-1/2^{2t-1}}.$$
\end{theorem}

(Here, we use the standard asymptotic notation: if $X$ is a set and $f,h:X \to \mathbb{R}^{+}$, we write $f =O(h)$ if there exists an absolute constant $C >0$ such that $f(x) \leq C h(x)$ for all $x \in X$.)

It would be of interest to narrow the gap between our upper and lower bounds on $g(m,t)$.

The proofs of Theorem \ref{thm:gmt_ub} and of the first part of Theorem \ref{thm:gmt_lb} (i.e.\ the $t > 2$ case) are very similar to those used in \cite{GUNDERSON1998} to prove Theorem \ref{thm:gunder} above. The proof of our lower bound for $t > 2$ consists of a probabilistic construction very similar to that of Gunderson and R\"odl in \cite{GUNDERSON1998}; we have to choose a random set of slightly lower density as we must avoid the middle slice of a Hilbert cube, as opposed to an entire Hilbert cube. At first sight, it might seem that an upper bound on $g(m,t)$ follows from the upper bound given in Theorem \ref{thm:gunder}, since a set which contains a Hilbert cube contains its middle layer, but one must make small changes to the proof in \cite{GUNDERSON1998} so as to ensure that the generators of the Hilbert cube we find (viz., the $x_i$ in the above definition), are distinct. In fact, we find an entire Hilbert cube generated by these $x_i$, not just the middle layer of such, so there may be some room for improvement.

We give an explicit construction for the $t=2$ case in Theorem $\ref{thm:gmt_lb}$: this construction outperforms the probabilistic one in the case $t=2$. The probabilistic approach does not sample directly from the ground set but instead from a large 3-AP-free set. This choice is critical in (optimizing) the argument, since it destroys additive structure in the (still large) set we obtain. 

Having introduced the background and general structure of our approach, we briefly consider the asymmetric version of the daisy problem. For $t \geq 1$ and $s \geq t+1$, an {\em $r$-uniform $(s,t)$-daisy} $\mathcal{D}_r^{s,t}$ is defined to be a set of the form
$$\{S \cup T \ :  T\subset U, \ |T|=t  \}$$
for some sets $S,U$ with $|S|=r-t$, $|U|=s$ and $S \cap U = \emptyset$. (So $\mathcal{D}_{r}^{2t,t} := \mathcal{D}_r^{t}$.) In this language, we have so far only considered $(2t,t)$-daisies. Conjecture \ref{conj:daisy} immediately implies the analog for asymmetric daisies ($s \neq 2t$), since an $(s,t)$-daisy is contained in a $(2\max\{t,s-t\},\max\{t,s-t\})$-daisy. In \cite{BOLLOBS2011} Bollob\'as, Leader, and Malvenuto focus on the symmetric case ($s=2t$), and in this article we do the same. However, the proofs in the sequel may be modified in the obvious way to obtain the following asymmetric analogue of Theorem \ref{thm:main_result}.

\begin{theorem}
For each $t \geq 2$ and $s \geq 4$ with $(s,t) \neq (4,2)$, we have 
\begin{equation*}
    \pi(D_r^{s,t}) \geq r^{-\frac{2s-2}{\binom{s}{t}-1}-O(1/\sqrt{\log r})}.
\end{equation*}
\end{theorem}

The remainder of this paper is structured as follows. In Section \ref{sec:transfer}, we prove Theorem \ref{thm:main_result} using Theorem \ref{thm:gmt_lb}. In Section \ref{sec:gmt}, we prove Theorems \ref{thm:gmt_lb} and \ref{thm:gmt_ub}. Note that we do not require the latter in our study of the Tur\'an density of daisies, but it may be of independent interest.

\section{The Proof of Theorem \ref{thm:main_result}}\label{sec:transfer}
\begin{proof}
In proving Theorem \ref{thm:main_result}, by an appropriate choice of $c$ (and of the absolute constant implicit in the big-O notation), we may clearly assume that $r \geq 8$. For $r,t \in \mathbb{N}$ with $r \geq 8$ and $n$ sufficiently large depending on $r$ and $t$, we proceed to construct a $\mathcal{D}_r^t$-free family of $r$-element subsets of $[n]$. We may assume that $n \geq 2r^2$. Let $L$ be a prime number such that $r^2 < L < 2r^2$ (such exists, by Bertrand's postulate). We use a `partite' construction, partitioning $[n]$ into $L$ blocks, and then taking only $r$-sets containing at most one element from each block. Formally, the block of vertex $i$ will be denoted by a variable $x_i$, where for $i \in [n]$ we set $x_i = \floor{L\frac{i-1}{n}}$; note that $0 \leq x_i < L$ for each $i \in [n]$. By Definition \ref{def:g}, there exists a set $R\subset \Z_L$ of size $|R| = g(L,t)$ with the property that for any $X \subset \Z_L$ with $|X|=2t$ and for any $x_0 \in \Z_L$, we have $x_0+C(X) \not \subset R$. Define a family $\F_{R}\subset \binom{[n]}{r}$ by
\begin{equation*}
\F_{R}=\left\{S \in \binom{[n]}{r} \ \mid \ \sum_{i \in S} x_i \in R \quad \text{and} \quad  (\forall  i,j \in S)  (x_i=x_j \Rightarrow i=j)  \right\}.
\end{equation*}
First, we check that $\F_{R}$ is $\D_r^t$-free. Indeed, suppose for the sake of a contradiction that $\F_{R}$ contains a daisy $\mathcal{D}=\{S_0 \cup T \ :  \ T\subset U, \ |T|=t  \}$, where $S_0,\ U\subset [n]$ with $|S_0|=r-t$, $|U|=2t$ and $S_0 \cap U = \emptyset$. Let $x_0:= \sum\limits_{i \in S_0}x_i$. By the above property of $R$, the $(2t)$-element set
\begin{equation*}
X=\{x_i \ : \ i \in U \} \subset \mathbb{Z}_L
\end{equation*}
must satisfy $x_0+C(X) \not \subset R$, and therefore there is a $t$-sum, indexed by $T=\{i_1,i_2,\dots,i_t\}\subset U$, say, such that
\begin{equation*}
x_0+x_{i_1}+x_{i_2}+\dots +x_{i_t}=x_0+\sum_{i \in T} x_i \not \in R.
\end{equation*}
It follows that $S:=S_0 \cup T \notin \F_{R}$ and therefore $\mathcal{D} \not \subset \mathcal{F}_R$, a contradiction, as required.

Now to finish the proof of Theorem \ref{thm:main_result} we bound $|\F_R|$ from below. First note that there are at least $\frac{1}{2}\binom{n}{r}$ sets $S \in \binom{[n]}{r}$ with $x_i\neq x_j \text{ for all } i \neq j, \ i,j \in S$. Indeed, choose a set $S$ uniformly at random from ${[n] \choose r}$. Since the probability that a uniformly random two-element subset $\{i,j\}$ of $[n]$ has $x_i=x_j$ is at most $1/L$, we have
\begin{align*}
\P(x_i=x_j \text{ for some } i \neq j, \ i,j \in S)
&\leq (1/L){n \choose 2}{n-2 \choose r-2}/{n \choose r}\\
& = r(r-1)/(2L)\\
& \leq 1/2.
\end{align*}
The family $\F = \F_R$, defined above, is $\D_r^t$-free even if the set $R$ is replaced by a translate $R_a:=R+a$ for some $a \in \Z_L$. Averaging over all such translates yields some translate $R_a$ of $R$ such that $|\F_{R_a}| \geq \frac{1}{2}\binom{n}{r}\frac{|R|}{L}$, and therefore
\begin{align*}
\pi(D_r^t) \geq \frac{g(L,t)}{2L}.
\end{align*}
Now we may apply Theorem \ref{thm:gmt_lb}, recalling that $r^2 < L < 2r^2$ is prime. When $t=2$ and $L=m \geq 64$ (which follows from $r \geq 8$), we have
\begin{equation*}
\pi(D_r^2) \geq \frac{g(L,2)}{2L} \geq \frac{\sqrt{L}}{16L}  \geq \frac{\sqrt{r^2}}{32r^2} = \frac{1}{32 r}
\end{equation*}
and when $t>2$ we have
\begin{equation*}
\pi(D_r^t) \geq \frac{\ L^{1-\frac{2t-1}{\binom{2t}{t}-1}-O(1/\sqrt{\log L})}}{L}\geq \ (r^2)^{-\frac{2t-1}{\binom{2t}{t}-1}-O(1/\sqrt{\log r})} = r^{-\frac{4t-2}{\binom{2t}{t}-1} -O(1/\sqrt{\log r})},
\end{equation*}
as required.
\end{proof}

\section{Bounds on $g(m,t)$}\label{sec:gmt}

The focus of this section is the analysis of $g(m,t)$.
\subsection{Proof of Theorem \ref{thm:gmt_lb}}
\begin{proof}
First assume $t \geq 3$. In this case, we use the idea of Gunderson and R\"odl (in \cite{GUNDERSON1998}) of passing to a fairly dense subset of $\mathbb{Z}_m$ which is free of 3-term arithmetic progressions; this `destroys' a lot of the additive structure we want to avoid.

By the well-known construction of Behrend in \cite{Behrend1946}, there exists a set $R_0 \subset [\floor{m/2}]$ with $|R_0| = m^{1-\gamma(m)}$ for $\gamma(m) := \frac{4}{\sqrt{\log{(m/2)}}}$ (here, and elsewhere, $\log$ denotes the natural logarithm), such that $R_0$ contains no $3$-term arithmetic progression. Let $R_1 \subset \Z_m$ be the natural embedding of $R_0$ into $\Z_m$. Then $R_1$ also contains no $3$-term arithmetic progression.
Set 
\begin{equation*}
p=\begin{cases}
\frac{1}{8}m^{-\frac{2t-1+\gamma(m)}{\binom{2t}{t}-1}} \quad &\text{if }m \text{ is prime,}\\
\frac{1}{8}m^{-\frac{2t+\gamma(m)}{\binom{2t}{t}-1}}\quad  &\text{otherwise,}
\end{cases}
\end{equation*}
and choose a set $R_2 \subset R_1$ by including each element of $R_1$ independently at random with probability $p$. A standard Chernoff bound (for example Theorem 4.5 in \cite{mm}) yields
\begin{equation}\label{eqn:r_bnd}
\P(|R_2| \leq |R_1| p/2) \leq e^{-|R_1|p/8}.
\end{equation}
Define the random variable
$$Y = |\{x_0+C(X):\ x_0 \in \mathbb{Z}_m,\ X \subset \mathbb{Z}_m,\ |X|=2t,\ x_0+C(X) \subset R_2\}|.$$
Since $R_2$ does not contain any $3$-term arithmetic progressions, for any set of the form $x_0+C(X)$ lying within $R_2$, we must have $|x_0+C(X)| = |C(X)|= \binom{2t}{t}$. Indeed, suppose for a contradiction that $X = \{x_1,\dots,x_{2t}\}$ is a $(2t)$-element subset of $\mathbb{Z}_m$ with $x_0+C(X) \subset R_2$, where $x_0 \in \mathbb{Z}_m$ and $|C(X)| < {2t \choose t}$. Then there exist two distinct $t$-element subsets of $X$, $\{x_{i_1},\ldots,x_{i_t}\} = S_1$ and $\{x_{i'_1},\ldots,x_{i'_t}\} =S_2$ say, such that
\begin{equation*}
x_{i_1}+\dots+x_{i_t} = x_{i'_1}+\dots+x_{i'_t};
\end{equation*}
we may assume without loss of generality that $x_{i_1} \in S_1 \setminus S_2$ and $x_{i'_1} \in S_2 \setminus S_1$, so that $x_{i_1} \neq x_{i'_j}$ for all $j$ and $x_{i'_1} \neq x_{i_j}$ for all $j$. Then
\begin{equation*}
\{x_{i'_1}+x_{i_2}+\dots+x_{i_t}, x_{i_1}+x_{i_2}+\dots+x_{i_t},x_{i_1}+x_{i'_2}+\dots+x_{i'_t} \} \subset R_2
\end{equation*}
is a (nontrivial) $3$-term arithmetic progression in $R_2$, a contradiction.

We now proceed to bound $\E Y$ from above. In the case that $m$ is not prime we may crudely bound the number of possible sets of the form $x_0+C(X)$ from above by $m^{2t+1}$, which is the number of choices for $x_0,x_1,\dots,x_{2t} \in \Z_m$. If $m$ is prime then we may assume each such set has $x_0=0$, by translating each of $x_1,\dots,x_{2t}$ by $-t^{-1}x_0$, leaving only $m^{2t}$ choices. The probability that each fixed set of the form $x_0+C(X)$ lies in $R_2$ is of course $p^{{2t \choose t}}$. It follows that
\begin{equation*}
\E Y \leq \left.\begin{cases}m^{2t}p^{\binom{2t}{t}} \quad & \text{if }m \text{ is prime}\\
m^{2t+1}p^{\binom{2t}{t}} \quad &\text{otherwise}
\end{cases} \right\}\leq m^{1-\gamma(m)} \tfrac{p}{8}.
\end{equation*}
It follows from Markov's inequality that
\begin{equation}\label{eqn:y_bnd}
\P(Y \geq m^{1-\gamma(m)} p/4) \leq 1/2.
\end{equation}
Combining \eqref{eqn:r_bnd} and \eqref{eqn:y_bnd}, we obtain
\begin{equation}\label{eqn:unn_bnd}
\P\left(|R_2|> m^{1-\gamma(m)} p/2 \text{ and } Y < m^{1-\gamma(m)} p/4 \right) \geq 1-e^{-m^{1-\gamma(m)}p/8}-\tfrac{1}{2}.
\end{equation}
Clearly, for any $t \geq 2$ and $m$ sufficiently large depending on $t$, we have $1-\gamma(m)-\frac{\gamma(m)+2t}{\binom{2t}{t}-1}>0$, so for large enough $m$, the probability in \eqref{eqn:unn_bnd} is positive, and therefore there exists a set $R_2 \subset \Z_m$ with $|R_2|> m^{1-\gamma(m)}p/2$ and $Y <  m^{1-\gamma(m)}p/4$. Now for each set of the form $x_0+C(X) \subset R_2$ for $(x_0,X)=(x_0,\{x_1,\dots,x_{2t}\})$ we remove a single element from $R_2$, chosen arbitrarily from $x_0+C(X)$. The total number of elements deleted from $R_2$ is at most $Y < m^{1-\gamma(m)}p/4$ and we are still left with 
\begin{equation*}
|R_2|-Y \geq m^{1-\gamma(m)} \tfrac{p}{4}=\begin{cases}\tfrac{1}{32} m^{1-\gamma(m)-\frac{2t-1+\gamma(m)}{\binom{2t}{t}-1}} \quad & \text{if }m \text{ is prime}\\
\tfrac{1}{32} m^{1-\gamma(m)-\frac{2t+\gamma(m)}{\binom{2t}{t}-1}} \quad &\text{otherwise}
\end{cases} 
\end{equation*}
elements, finishing the proof of the first statement of Theorem \ref{thm:gmt_lb}.

Finally, in the case $t=2$, we give an algebraic construction that improves upon the random one. First, we recall the definition of a Sidon set.
\begin{definition}
A {\em Sidon set} in an Abelian group $G$ is a subset $S \subset G$ such that the only solutions to the equation $a+b=c+d$ with $a,b,c,d \in S$, are the trivial ones (meaning, those with $\{a,b\} = \{c,d\}$).
\end{definition}
It follows from the classical construction of Singer \cite{Singer1938} that for any prime $p$ there is a Sidon set of size $p+1$ inside $\Z_{p^2+p+1}$. Assume that $m \geq 64$ and let $p$ be a prime with $\sqrt{m}/8 \leq p \leq \sqrt{m}/4$ (such is furnished by Bertrand's postulate). Let $R_0$ be a Sidon set of size at least $\sqrt{m}/8$ inside $\Z_{p^2+p+1}$. The image $R$ of $R_0$ under the natural inclusion map  from $\Z_{p^2+p+1}$ to $\Z_{m}$ is a Sidon set in $\mathbb{Z}_m$ (here, we use $p^2+p+1\leq m/16+\sqrt{m}/4+1 < m/2$).
Now we will show that for any $x_0$ and $X=\{x_1,x_2,x_3,x_4\} \in {\Z_m \choose 4}$ we have $x_0+C(X) = x_0+\{x_1+x_2,x_1+x_3,x_1+x_4,x_2+x_3,x_2+x_4,x_3+x_4\} \not \subset R$. Suppose for a contradiction that $x_0+C(X) \subset R$; then
\begin{equation*}
(x_0+x_1+x_2)+(x_0+x_3+x_4)=(x_0+x_1+x_3)+(x_0+x_2+x_4)
\end{equation*}
and each term in brackets is an element of $R$. Since $R$ is a Sidon set, this implies $x_2=x_3$ or $x_1=x_4$, contradicting the fact that the $x_i$ are distinct. We have $|R| \geq \sqrt{m}/8$ and therefore we are done in the case $t=2$.
\end{proof}
\begin{rem}
We were not able to generalize the Sidon set approach to $t>2$, hence our reliance on the probabilistic construction avoiding middle layers of Hilbert cubes. There are some improvements upon Behrend's construction, for example by Elkin in \cite{Elkin} and Green and Wolf in \cite{GreenWolf}. Using these yields, for $t\geq3$, a slightly better error-term in the exponent of $r$, but this does not affect the main term in the exponent.
\end{rem}
\subsection{Proof of Theorem \ref{thm:gmt_ub}}
\begin{proof}
We begin with a quick calculation.
\begin{lemma}\label{lem:binom_algebra}
If $m,d$, and $b$ are positive real numbers with $m \geq d+1$ and $b \geq \max\{\frac{\sqrt{m}+\sqrt{d}}{2\sqrt{d}},4d+1\}$, then $\frac{\binom{b}{2}-db}{m-d} \geq \frac{b^2}{4(\sqrt{m}+\sqrt{d})^2}$.
\end{lemma}
\begin{proof}
Since
\begin{equation*}
    \frac{1}{b} \leq \frac{2\sqrt{d}}{\sqrt{m}+\sqrt{d}},
\end{equation*}
we have
\begin{equation*}
    \frac{b-1}{b} \geq \frac{\sqrt{m}-\sqrt{d}}{\sqrt{m}+\sqrt{d}}.
\end{equation*}
Since $b \geq 4d+1$ we have $\binom{b}{2}-db \geq \frac{b(b-1)}{4}$, and therefore
\begin{equation*}
    \binom{b}{2}-db \geq \frac{b^2}{4}\frac{\sqrt{m}-\sqrt{d}}{\sqrt{m}+\sqrt{d}}.
\end{equation*}
Dividing by $m-d$ yields the result.
\end{proof}

We may now obtain our upper bound on $g(m,t)$. Let $A \subset \Z_m$ such that 
\begin{equation}
\label{eq:Alb}
    |A| \geq 4^{1-1/2^{2t}}(\sqrt{m}+\sqrt{2t})^{2-1/2^{2t-1}}.
\end{equation}
We will show that $x_0+C(X) \subset A$ for some $x_0 \in \mathbb{Z}_m$ and $X \in {\mathbb{Z}_m \choose 2t}$. For $x_1,\dots,x_{d} \in \Z_m$ and $A \subset \Z_m$ we define $A_{x_1} := A \cap (A-x_1)$, $A_{x_1,x_2} := A_{x_1} \cap (A_{x_1}-x_2)$, and more generally,
\begin{equation*}
    A_{x_1,\dots,x_{d-1},x_{d}} := A_{x_1,\dots,x_{d-1}} \cap (A_{x_1,\dots,x_{d-1}}-x_{d}).
\end{equation*}
Then $A_{x_1,\dots,x_{d}} = \{x \in \mathbb{Z}_m:\ x+\sum_{i \in I} x_i \in A \ \forall \ I \subset [d] \}$ and so $A$ will contain a  translate of $C(\{x_1,\dots,x_{2t}\})$ if $|A_{x_1,\dots,x_{2t}}| \geq 1$. We will find these $x_i$ inductively, using the following claim.
\begin{cla}
\label{claim:final}
Provided $m$ is sufficiently large depending on $d$, for each $0\leq d \leq 2t$ there exist $d$ distinct elements $x_1,\dots,x_{d} \in \Z_m$ such that
\begin{equation*}
    |A_{x_1,\dots,x_{d}}| \geq \frac{|A|^{2^{d}}}{4^{2^{d}-1}(\sqrt{m}+\sqrt{2t})^{2^{d+1}-2}}.
\end{equation*}
\end{cla}
(Note that when $d=0$, the left-hand side of the above is defined to be $|A|$.)
\begin{proof}[Proof of Claim]
The proof is by induction on $d$ (with base case $d=0$, for which the claim holds trivially). Suppose the claim holds for $d$ for elements $x_1,\dots,x_d$. Since every pair of elements in $A_{x_1,\dots,x_{d}}$ uniquely determine a difference $y$, we have
\begin{equation*}
\sum_{y \in \Z_m} |A_{x_1,\dots,x_{d},y}| \geq \binom{|A_{x_1,\dots,x_{d}}|}{2}.
\end{equation*}
Therefore, when forbidding $y$ to assume the values $\{x_1,\dots,x_d\}$, we crudely obtain
\begin{equation*}
\sum_{y \in \Z_m \setminus \{x_1,\ldots,x_{d}\}} |A_{x_1,\dots,x_{d},y}| \geq \binom{|A_{x_1,\dots,x_{d}}|}{2} - d|A_{x_1,\dots,x_{d}}|.
\end{equation*}
By averaging over $y \in \Z_m \setminus \{x_1,\ldots,x_{d}\}$, there exists $y' \in \Z_m \setminus \{x_1,\ldots,x_{d}\}$ such that
$$|A_{x_1,\dots,x_{d},y'}| \geq \frac{\binom{|A_{x_1,\dots,x_{d}}|}{2} - d|A_{x_1,\dots,x_{d}}|}{m-d}.$$
We now wish to apply Lemma $\ref{lem:binom_algebra}$ with $b=|A_{x_1,\dots,x_{d}}|$. The hypotheses that $m \geq d+1$ and $b \geq 4d+1$ are satisfied for $m$ large enough (depending on $t$), so to apply Lemma $\ref{lem:binom_algebra}$ it remains only to check that $b \geq \frac{\sqrt{m}+\sqrt{d}}{2\sqrt{d}}$, which follows from our inductive hypothesis and our lower bound (\ref{eq:Alb}) on $|A|$, as well as the fact that $d<2t$ in the inductive step:
\begin{align*}
    b & = |A_{x_1,\dots,x_{d}}|\\
    &\geq \frac{|A|^{2^{d}}}{4^{2^{d}-1}(\sqrt{m}+\sqrt{2t})^{2^{d+1}-2}}\\
    &\geq \frac{4^{2^{d}-2^{d-2t}}(\sqrt{m}+\sqrt{2t})^{2^{d+1}-2^{d-2t+1}}}{4^{2^{d}-1}(\sqrt{m}+\sqrt{2t})^{2^{d+1}-2}}\\
    &\geq 4^{1-2^{d-2t}}(\sqrt{m}+\sqrt{2t})^{2-2^{d-2t+1}}\\
    &\geq \frac{\sqrt{m}+\sqrt{d}}{2\sqrt{d}}.
\end{align*}
Hence, applying Lemma $\ref{lem:binom_algebra}$ we have
\begin{align*}
    |A_{x_1,\dots,x_{d},y'}| &\geq \frac{\binom{|A_{x_1,\dots,x_{d}}|}{2} - d|A_{x_1,\dots,x_{d}}|}{m-d}\\
    &\geq \frac{|A_{x_1,\dots,x_{d}}|^2}{4(\sqrt{m}+\sqrt{d})^2}\\
    &\geq \frac{|A|^{2^{d+1}}}{4^{2^{d+1}-2+1}(\sqrt{m}+\sqrt{d})^{2^{d+2}-4+2}},
\end{align*}
as required. We may therefore choose $x_{d+1}=y'$.
\end{proof}
Applying Claim \ref{claim:final} with $d=2t$, and using our lower bound (\ref{eq:Alb}) on $|A|$, we obtain distinct $x_1,\ldots,x_{2t} \in \mathbb{Z}_m$ such that $|A_{x_1,\dots,x_{2t}}| \geq 1$, completing the proof of Theorem \ref{thm:gmt_ub}.
\end{proof}

\subsubsection*{Acknowledgement}
We are very grateful to Victor Souza for bringing the reference \cite{GUNDERSON1998} to our attention.

\end{document}